\documentclass[11pt]{amsart}
\usepackage{amssymb, amsthm, amsmath, gensymb, dsfont}
\usepackage{graphicx, comment}
\usepackage[small]{caption}
\usepackage{subcaption}
\usepackage{epsfig}
\usepackage{tikz, pgfplots, float}
\usepackage{setspace}
\usepackage{amsfonts}

\usepackage[utf8]{inputenc}
\usepackage{fullpage}
\usepackage{framed}

\usepackage{enumerate}
\usepackage{url}
\usepackage[breaklinks]{hyperref}
\usepackage{cleveref}
\hypersetup{
	colorlinks = true, 
	urlcolor = cyan, 
	linkcolor = teal, 
	citecolor = cyan 
}

\newtheorem{theorem}{Theorem}

\crefname{claim}{claim}{claims}

\title{Remarks on a theorem of Erd\H{o}s and Szemer\'{e}di}

\author{Dingyuan Liu}
\address{Dingyuan Liu \newline Karlsruhe Institute of Technology, Englerstraße 2, D-76131 Karlsruhe, Germany}
\email{liu@mathe.berlin}

\begin{document}
\maketitle

\begin{abstract}
Given a graph $G$ and a real $\varepsilon>0$, an edge-coloring of $G$ is called \textit{$\varepsilon$-balanced} if each color appears on at least an $\varepsilon$-fraction of the edges in $G$. A classical result of Erd\H{o}s and Szemer\'{e}di asserts that if a $2$-edge-coloring of a complete graph $K_n$ is not $\varepsilon$-balanced for some $0<\varepsilon\leq1/2$, then there exists a large monochromatic clique. This theorem has been used extensively in Ramsey-type arguments, as it allows one to focus on reasonably balanced colorings. However, in its original formulation the dependence between $n$ and $\varepsilon$ was left implicit, occasionally leading to inaccurate applications. In this short note, we revisit the Erd\H{o}s--Szemer\'{e}di theorem and specify all parameter dependencies.
\end{abstract}

\section{Introduction}
A central theme in Ramsey theory is the emergence of large homogeneous structures in edge-colored complete graphs. One classical manifestation of this phenomenon is a result of Erd\H{o}s and Szemer\'{e}di~\cite{erdos1972ramsey}, which asserts that if the edges of a complete graph are colored in an unbalanced way, then one shall necessarily find a large monochromatic clique. This theorem has become a standard tool in a variety of Ramsey-type arguments, where it is often used to justify a reduction to the case of relatively balanced colorings.

Formally, given a graph $G$ and a real $\varepsilon>0$, an edge-coloring of $G$ is called \textit{$\varepsilon$-balanced} if each color appears on at least an $\varepsilon$-fraction of the edges in $G$. The Erd\H{o}s--Szemer\'{e}di theorem states that if a $2$-edge-coloring of an $n$-vertex complete graph $K_n$ is not $\varepsilon$-balanced for some $0<\varepsilon\leq1/2$, then there exists a monochromatic clique of size at least $\frac{C\log{n}}{\varepsilon\log(1/\varepsilon)}$, where $C>0$ is an absolute constant. In its original statement, it was assumed implicitly that the parameter $\varepsilon$ is fixed and the number of vertices $n$ is sufficiently large. This level of precision is sufficient for asymptotic applications, yet later work has at times invoked the theorem for arbitrary values of $n$ and $\varepsilon$.

This note aims to give a precise version of the Erd\H{o}s--Szemer\'{e}di theorem, where all parameter dependencies are explicit. Observe that a $2$-edge-coloring of $K_n$ is not $\varepsilon$-balanced if and only if the subgraph $G\subseteq K_n$ consisting of the edges of the majority color has more than $(1-\varepsilon)\binom{n}{2}$ edges. By letting $\varepsilon=1/k$, the theorem can be stated in the following equivalent form. All logarithms are taken to base $2$.

\begin{theorem}[{\cite[Theorem~2]{erdos1972ramsey}}]
\label{main}
Let $0<C\leq0.01$ be a constant. For every integer $n\geq3$ and every real $k\geq2$ with $k\leq \frac{n}{3C}$, if an $n$-vertex graph $G$ has at least $(1-1/k)\binom{n}{2}$ edges, then $G$ contains a clique or an independent set of size at least $\frac{Ck\log{n}}{\log{k}}$.
\end{theorem}
Note that if $k>n/C^2$, then the value of $\frac{Ck\log{n}}{\log{k}}$ would exceed $n$. Thus, the dependence between $n$ and $k$ in~\Cref{main} is essentially optimal.

\section{Proof of~\Cref{main}}
Given positive integers $s$ and $t$, recall that the \emph{Ramsey number} $R(s,t)$ is the smallest integer $r$ such that any $r$-vertex graph contains a clique of size $s$ or an independent set of size $t$.
\begin{proof}
Fix any $0<C\leq0.01$. Let $n\geq3$ be an integer and $2\leq k\leq \frac{n}{3C}$.

\quad\\\textbf{Case 1:} $k\leq100$.

Plugging $s=t=\frac{Ck\log{n}}{\log{k}}$ in the upper bound
\begin{equation}
\label{upper_bound}
R(\lceil{s}\rceil,\lceil{t}\rceil)\leq\binom{\lceil{s}\rceil+\lceil{t}\rceil-2}{\lceil{s}\rceil-1}\leq\binom{s+t}{\lfloor{s}\rfloor}
\end{equation}
by Erd\H{o}s and Szekeres~\cite{erdos1935ramsey}, we obtain that $R(\lceil{s}\rceil,\lceil{t}\rceil)<2^{s+t}\leq4^{\frac{Ck\log{n}}{\log{k}}}<n$. Therefore, $G$ contains a clique or an independent set of size at least $\frac{Ck\log{n}}{\log{k}}$.

\quad\\\textbf{Case 2:} $k\geq\sqrt{n}$.

We have in this case $\frac{Ck\log{n}}{\log{k}}\leq2Ck$. Suppose that $G$ contains no clique of size $\lceil2Ck\rceil$, which also implies that $\lceil2Ck\rceil\geq3$. Then by Tur\'an's theorem~\cite{turan1941extremal}, the number of edges in $G$ is at most
\[\left(1-\frac{1}{\lceil2Ck\rceil-1}\right)\frac{n^2}{2}<\left(1-\frac{1}{2Ck}\right)\frac{n^2}{2}<\left(1-\frac{1}{k}\right)\binom{n}{2}-\left(\frac{1-2C}{2Ck}\cdot\frac{n^2}{2}-\frac{n}{2}\right)<\left(1-\frac{1}{k}\right)\binom{n}{2},\]
a contradiction. Note that in the last inequality we used $n\geq3Ck$ and $0<C\leq0.01$.

\quad\\\textbf{Case 3:} $100<k<\sqrt{n}$.

In this case we follow the Erd\H{o}s--Szemer\'edi argument. Let $W\subseteq V(G)$ be the set of vertices whose degree in $G$ is smaller than $(1-2/k)n$. Let $H$ be the induced subgraph of $G$ obtained by deleting all the vertices in $W$. As $|E(G)|\geq(1-1/k)\binom{n}{2}$, it holds that $|W|\leq n/2$, i.e., $|V(H)|\geq n/2$. Let $A\subseteq V(H)$ be a largest clique in $H$ and let $B=V(H)\setminus A$. Without loss of generality assume that $|A|<\frac{Ck\log{n}}{\log{k}}\leq\frac{k\log{n}}{100\log{k}}$.

\begin{center}
\textit{\textbf{Claim.} There are at most $n/5$ vertices in $B$, each having at most $(1-10/k)|A|$ neighbors in $A$.}
\end{center}

Indeed, otherwise the number of non-edges in $G$ between $A$ and $V(G)\setminus{A}$ is greater than $2n|A|/k$. Meanwhile, since every vertex in $A$ has degree at least $(1-2/k)n$ in $G$, the number of non-edges in $G$ between $A$ and $V(G)\setminus{A}$ is at most $2n|A|/k$, a contradiction. This proves the claim.

The number of subsets of $A$ of cardinality at most $10|A|/k$ is upper bounded by \[\frac{10|A|}{k}\binom{|A|}{10|A|/k}\leq\frac{10|A|}{k}\left(\frac{ek}{10}\right)^{10|A|/k}<\frac{\log{n}}{10\log{k}}\cdot2^{\frac{\log{n}}{10\log{k}}\cdot\log\frac{ek}{10}}<n^{1/3},\]
where in the first inequality we used the estimate $\binom{x}{y}\leq(ex/y)^y$ and in the second inequality we used $|A|<\frac{k\log{n}}{100\log{k}}$. Since $|V(H)|\geq n/2$ and $|A|<\frac{k\log{n}}{100\log{k}}<0.01n$, we have $|B|\geq 0.49n$. Namely, there are at least $|B|-n/5\geq n/4$ vertices in $B$, each having more than $(1-10/k)|A|$ neighbors in $A$. Since $\frac{n/4}{n^{1/3}}\geq\sqrt{n}$, by the pigeonhole principle, there is a subset $A'\subseteq A$ with $|A'|<10|A|/k$ and a subset $B'\subseteq B$ with $|B'|\geq\sqrt{n}$, such that every vertex in $B'$ is adjacent to all the vertices in $A\setminus A'$. Let $H[B']$ be the subgraph of $H$ induced by $B'$.

Observe that $H[B']$ contains no clique of size larger than $|A'|$. Indeed, if $H[B']$ contains a clique $Q$ with $|Q|>|A'|$, then $(A\setminus A')\cup Q$ would be a clique in $H$ which is larger than $A$, contradicting the fact that $A$ is a largest clique in $H$.

Therefore, $H[B']$ is a graph on at least $\sqrt{n}$ vertices without cliques of size $\lceil10|A|/k\rceil$. Applying the upper bound~\eqref{upper_bound} with $s=10|A|/k$ and $t=\frac{k\log{n}}{100\log{k}}>|A|>10s$, we have
\[R(\lceil{s}\rceil,\lceil{t}\rceil)\leq\binom{s+t}{\lfloor{s}\rfloor}\leq\left(\frac{10t}{s}\right)^{s}\leq2^{\frac{10|A|}{k}\cdot\log\frac{k^2\log{n}}{100|A|\log{k}}}<n^{1/10}.\]
In the last inequality we used $|A|<\frac{k\log{n}}{100\log{k}}$. Accordingly, there is an independent set in $H[B']$ (and hence in $G$) of size at least $t\geq\frac{Ck\log{n}}{\log{k}}$.
\end{proof}

\vspace{1em}
\noindent
\textbf{Acknowledgement.} Many thanks to Maria Axenovich for suggesting the writing of this note and for her thoughtful comments.

\end{document}